\documentclass[11pt]{amsart}
\usepackage{amssymb, enumerate, verbatim, fullpage}
\usepackage{versions}  

\usepackage{graphicx}

\newcommand{\Gr}[2]{\includegraphics[width=#2]{#1}}

\excludeversion{not_arXiv}
\includeversion{arXiv}

\newcommand{\tth}{^{\operatorname{th}}}
\newcommand{\rrd}{^{\operatorname{rd}}}
\newcommand{\ZZ}{\mathbb{Z}}
\newcommand{\QQ}{\mathbb{Q}}
\newcommand{\PP}{\mathbb{P}}

\newcommand{\FF}{\mathbb{F}}

\newcommand{\Aff}{\mathbb{A}}

\newcommand{\im}{\operatorname{im}}
\newcommand{\pc}[1]{Y^{\operatorname{pre}}\left(#1\right)}
\newcommand{\cpc}[1]{X^{\operatorname{pre}}\left(#1\right)}
\renewcommand{\bar}{\overline}
\newcommand{\jpc}[1]{J^{\operatorname{pre}}\left(#1\right)}
\newcommand{\Pic}{\operatorname{Pic}}

\newcommand{\be}{\begin{equation}}
\newcommand{\ee}{\end{equation}}

\DeclareMathOperator{\Spec}{Spec}

\newcommand{\eps}{\varepsilon}

\newtheorem{thm}{Theorem}[section]
\newtheorem*{thm*}{Theorem}
\newtheorem{lem}[thm]{Lemma}
\newtheorem{cor}[thm]{Corollary}
\newtheorem{prop}[thm]{Proposition}

\newtheorem*{conjecture*}{Conjecture}

\theoremstyle{remark}

\newtheorem{remark}[thm]{Remark}

\theoremstyle{definition}

\begin{document}

\title[Pre-images of the Origin]{On the Number of Rational Iterated Pre-images of the Origin Under Quadratic Dynamical Systems}

\author[Faber]{Xander Faber}
\address{
Department of Mathematics and Statistics \\
McGill University \\
Montr\'eal, QC  \\ Canada
}
\email{xander@math.mcgill.ca}

\author[Hutz]{Benjamin Hutz}
\address{
Department of Mathematics and Computer Science \\
Amherst College \\
Amherst, MA \\ USA
}
\email{bhutz@amherst.edu}

\author[Stoll]{Michael Stoll}
\address{
Mathematisches Institut \\
Universit\"at Bayreuth \\
95440 Bayreuth \\ Germany
}
\email{Michael.Stoll@uni-bayreuth.de}

\begin{abstract}
	For a quadratic endomorphism of the affine line defined over the rationals, we consider the problem of bounding the number of rational points that eventually land at the origin after iteration. In the article ``Uniform Bounds on Pre-Images Under Quadratic Dynamical Systems,'' by two of the present authors and five others, it was shown that the number of rational iterated pre-images of the origin is bounded as one varies the morphism in a certain one-dimensional family. Subject to the validity of the Birch and Swinnerton-Dyer conjecture and some other related conjectures for the L-series of a specific abelian variety and using a number of modern tools for locating rational points on high genus curves, we show that the maximum number of rational iterated pre-images is six. We also provide further insight into the geometry of the ``pre-image curves.''
\end{abstract}




\maketitle


\section{Introduction}

	Fix a rational number $c \in \QQ$ and define an endomorphism of the affine line by
	\[
		f_c: \Aff^1_{\QQ} \to \Aff^1_{\QQ}, \qquad f_c(x) = x^2 + c.
	\]
If we define $f_c^N$ to be the $N$-fold composition of the morphism $f_c$, and $f_c^{-N}$ to be the $N$-fold pre-image, then for $a \in \Aff^1(\QQ)$, the set of \textbf{rational iterated pre-images of $a$} is given by
	\begin{equation*} \label{Eqn: Preimages}
		\bigcup_{N \geq 1} f_c^{-N}(a)(\QQ)
			= \{x_0 \in \Aff^1(\QQ): f_c^N(x_0) = a \text{ for some $N \geq 1$}\}.
	\end{equation*}

Heuristically, finding iterated pre-images amounts to solving progressively more complicated polynomial equations, and so rational solutions should be a rarity.  The situation becomes more interesting as we vary $c$, which has the effect of varying the morphism $f_c$. A special case of the main theorem in \cite{FHIJMTZ} shows that, independent of $c$, there is a bound on the size of the set of rational iterated pre-images:
	
\begin{thm}[{\cite[Thm.~1.2 for $B=D=1$]{FHIJMTZ}}]
\label{Thm: FHIJMTZ}
	Fix a point $a \in \Aff^1(\QQ)$ and define the quantity
	\[ 	
		\kappa(a) = \sup_{c \in \QQ}
			\# \Bigg\{\bigcup_{N \geq 1} f_c^{-N}(a)(\QQ) \Bigg\}.
	\]
	Then $\kappa(a)$ is finite.
\end{thm}

	In the present paper we study Theorem~\ref{Thm: FHIJMTZ} in the special case $a=0$ and obtain several refinements. The proof of Theorem~\ref{Thm: FHIJMTZ} shows that when $a \in \QQ \smallsetminus \{-1/4\}$ and $N \geq 4$, there are only finitely many values $c \in \QQ$ for which $f_c^{-N}(a)(\QQ)$ is nonempty. So if we define
		\[
			\bar{\kappa}(a) = \limsup_{c \in \QQ}  \#  \Bigg\{\bigcup_{N \geq 1} f_c^{-N}(a)(\QQ) \Bigg\},
		\]
then one obtains a trivial upper bound $\bar{\kappa}(a) \leq 14$ for $a  \neq -1/4$ by counting the degrees of $f_c$, $f_c^2$, and $f_c^3$. We can substantially improve this for $a=0$:

\begin{thm}
\label{Thm: Kappa bar bound}
	One has $\bar{\kappa}(0) = 6$.
\end{thm}	
	
	The content of this result amounts to the following three statements: (1) For any $c \in \QQ$, the map $f_c$ has at most two rational third pre-images of the origin; (2) For any $c \in \QQ$, if $f_c$ has four rational second pre-images of the origin, then it has no rational third pre-image; and (3) There are only finitely many rational $c \in \QQ$ such that $f_c$ has a rational fourth pre-image. The first and second statements are proved in \S\ref{Sec: Second and Third Pre-images}, while the third statement follows from Faltings' Theorem applied to a particular curve of genus~5. 
	
	Note that the two values $c=0$ and $c=-1$ correspond to morphisms $f_c$ for which $0$ is periodic of period~$1$ and~$2$, respectively.  For these values of $c$, the origin has at least one rational $N\tth$ pre-image for arbitrary $N$. A byproduct of the proof of Theorem~\ref{Thm: Kappa bar bound} is the following conditional result. 

\begin{thm}
\label{Thm: Kappa bound}
	Suppose that for $c \in \QQ \smallsetminus \{0,-1\}$ the morphism $f_c$ admits no rational $4\tth$ pre-images of the origin. Then $\kappa(0) = 6$.
\end{thm}

    		
	Define an algebraic set $\pc{N,a}$ in the $(x,c)$-plane by the equation $f_c^N(x) = a$. The algebraic points $(x_0, c_0) \in \pc{N,a}(\bar{\QQ})$ are in bijection with the $N\tth$ pre-images $x_0 \in f_{c_0}^{-N}(a)$.  For generic~$a$, $\pc{4,a}$ is a nonsingular curve of genus~5; the finiteness of its set of rational points played a key role in the proof of Theorem~\ref{Thm: FHIJMTZ}.  We apply the method of Chabauty and Coleman to determine this set.

\begin{thm}
\label{Thm: Bound for X(4,0)}
	Let $\jpc{4,0}$ denote the Jacobian of the complete nonsingular curve birational to $\pc{4,0}$. Assume the rank of $\jpc{4,0}(\QQ)$ is $3$. Then there is no value $c \in \QQ \smallsetminus \{0, -1\}$ such that $f_c$ has a rational $4\tth$ pre-image of the origin, and hence $\kappa(0) = 6$.
\end{thm}
	
	We remark that it is fairly easy to see that the rank of~$\jpc{4,0}(\QQ)$ is at least~$3$. With the currently available methods, it seems to be practically impossible to give an unconditional proof that the rank of~$\jpc{4,0}(\QQ)$ actually equals~$3$. We can, however, compute sufficiently many terms of the L-series of this Jacobian so that (assuming standard conjectures on L-series) we can numerically evaluate the L-series and its first few derivatives at the point~$s = 1$. The conjecture of Birch and Swinnerton-Dyer in its generalized form for abelian varieties then  provides us with an upper bound for the rank. Doing the computations, we obtain the following result.

\begin{thm}
\label{Thm: BSD bound for the rank}
	Assume the standard conjectures (analytic continuation and functional equation) for the L-series of~$\jpc{4,0}$. Then the third derivative of this L-series at $s = 1$ does not vanish. If in addition the (weak form of the) conjecture of Birch and Swinnerton-Dyer is satisfied for~$\jpc{4,0}$, then the rank of~$\jpc{4,0}(\QQ)$ is~$3$, and hence $\kappa(0) = 6$.
\end{thm}
	
	We present these conditional results for two reasons. First, by working with the curves $\pc{N,a}$, our problem can be reduced to the classical Diophantine pastime of finding rational points on a curve of high genus. Our setting provides a nice example on which to illustrate the use of several modern tools for finding rational points: the Mordell-Weil group of the Jacobian, the method of Chabauty-Coleman, and the Weil conjectures \cite{FPS, Poonen, Stoll_Independence_2006, Stoll3, vanluijk}. Second, in order to carry out these calculations we produce an explicit quasi-projective embedding of the curve $\pc{N,a}$. It is our hope that the simple nature of this embedding will be of assistance in future studies of the arithmetic of pre-images.

	In the next section, we recall the analogy between one-dimensional dynamical systems and elliptic curves; it often provides inspiration for dynamical research. In the elliptic curve setting, the result analogous to Theorem~\ref{Thm: FHIJMTZ} is Manin's result on $m$-power torsion \cite{ManinTorsion}. 
In \S\ref{Sec: Geometry} we briefly summarize the necessary background on pre-image curves developed in \cite{FHIJMTZ}. The entirety of \S\ref{Sec: Alternate} is devoted to exhibiting properties and consequences of a useful projective embedding of the curve $\pc{N,a}$. In \S\ref{Sec: Second and Third Pre-images} we prove Theorems~\ref{Thm: Kappa bar bound} and~\ref{Thm: Kappa bound}, and in \S\ref{Sec: Arithmetic} we turn to the arithmetic and geometric study of the Jacobian $\jpc{4,0}$ in order to prove Theorem~\ref{Thm: Bound for X(4,0)}. In \S\ref{Sec: BSD} we explain the $L$-series computations needed to deduce Theorem~\ref{Thm: BSD bound for the rank}. 
One step in these computations consists in showing that the wild part of the conductor at~2 of $\jpc{4,0}$, which has totally unipotent reduction at~2, vanishes. This involves the consideration of the special fibers of regular models of $\cpc{4,0}$ over certain tamely ramified extensions of~$\QQ_2$.
The final three sections make heavy use of the algebra and number theory system \textit{Magma} \cite{magma}.
\begin{arXiv} A Magma script for the relevant computations is included with this arXiv distribution. \end{arXiv} \begin{not_arXiv} A Magma script for the relevant computations is included with the arXiv distribution of this article \cite{Faber_Hutz_Stoll_Origin_arXiv_2010}. \end{not_arXiv}
	

\section{The elliptic curve analogy}
	Many phenomena in the arithmetic theory of dynamical systems have strong analogues in the theory of abelian varieties. For a fixed positive integer $m$ and an elliptic curve $E/\QQ$, denote the multiplication-by-$m$ map by $[m]: E \to E$. Let $\mathcal{O}$ be the origin for the group law on $E(\QQ)$. The key to the analogy is the following fact: the set of rational iterated pre-images of $\mathcal{O}$ by the morphism $[m]$ is precisely the $m$-power torsion subgroup of the Mordell-Weil group
	\[
		E[m^{\infty}](\QQ) = \bigcup_{N\geq 1} [m]^{-N}(\mathcal{O})(\QQ).
	\]
To continue the analogy, we replace our family of quadratic dynamical systems with the family of all elliptic curves and ask if there exists a uniform bound on $E[m^{\infty}](\QQ)$ as we vary the elliptic curve $E$.  The answer is given by Mazur's uniformity theorem for torsion in the Mordell-Weil group $E_{\operatorname{tors}}(\QQ) \subset E(\QQ)$. The following weak form of Mazur's theorem parallels Theorem~\ref{Thm: FHIJMTZ}. 
	
\begin{thm}[\cite{Mazur_Torsion_1978}]
	Let $E/\QQ$ be an elliptic curve. Then $\#E_{\operatorname{tors}}(\QQ) \leq 16$. In particular, if we define the
	quantity
		\[
			\kappa' = \sup_{E / \QQ}
			\#\left\{\bigcup_{N\geq 1}  [m]^{-N}(\mathcal{O})(\QQ)\right\},
		\]
	then $\kappa'$ is finite.
\end{thm}

We have chosen not to give an explicit value for $\kappa'$ in the statement above in order to draw out the analogy with Theorem~\ref{Thm: FHIJMTZ}. In fact, $\kappa'=16$ is an optimal choice since there exist elliptic curves $E/\QQ$ with torsion subgroup isomorphic to $\ZZ/2\ZZ \times \ZZ / 8\ZZ$.

The second statement in the theorem above, which says that the rational $m$-power
torsion is universally bounded for elliptic curves over~$\QQ$, is originally due
to Manin~\cite{ManinTorsion}, but without an explicit bound.

	
\section{The geometry of $\pc{N,a}$}
\label{Sec: Geometry}

Here we summarize the necessary geometric theory of pre-image curves developed in \cite{FHIJMTZ}. Most of the purely geometric results in \cite{FHIJMTZ} assume that the base field is algebraically closed. We will adapt the statement of Theorem~\ref{Thm: Irreducible Genus} below to account for this disparity. It is an adjustment in viewpoint; no additional proof is necessary.

Let $k$ be a field of characteristic different from~$2$.  As in the introduction, we define a morphism $f_c: \Aff^1_k \to \Aff^1_k$ for any $c \in k$ by the formula
	\[
		f_c(x) = x^2 + c.
	\]
We could view $f_c$ as an endomorphism of $\PP^1_k$, but the point at infinity is totally invariant for this type of morphism, and hence dynamically uninteresting. Fix a basepoint $a \in k$ and a positive integer $N$. Define an algebraic set
	\[
		\pc{N,a} = V\left(f_c^N(x) - a\right) \subset \Aff^2_k = \Spec k[x,c].
	\]
If $\pc{N,a}$ is geometrically irreducible, we define the \textbf{$N\tth$ pre-image curve}, denoted $\cpc{N,a}$, to be the unique complete nonsingular curve birational to $\pc{N,a}$. When we say a curve $C/k$ is nonsingular, we mean that it is nonsingular after extending scalars to the algebraic closure.\footnote{It is a standard fact in algebraic geometry that such a curve exists and is unique when $k$ is algebraically closed. Uniqueness follows in general from the fact that the curve is proper, but existence is a little trickier in the case of arbitrary $k$. See Corollary~\ref{Other fields} for a direct proof of existence in our case.} Recall also that the gonality of a curve $C/k$ is the minimum degree of a nonconstant morphism $C \to \PP^1$ (defined over $k$).

\begin{thm}[{\cite[Cor.~2.4, Thm.~3.2, \& Thm.~3.6]{FHIJMTZ}}]
\label{Thm: Irreducible Genus}
	Let $a \in k$, and let $N \geq 1$ be an integer for which $\pc{N,a}$ is nonsingular.
	Then $\pc{N,a}$ is geometrically irreducible. Moreover, $\cpc{N,a}$ is geometrically
	irreducible and has genus $(N-3)2^{N-2} + 1$. If $N \geq 2$, then $\cpc{N,a}$ has gonality $2^{N-2}$.
\end{thm}

We end this section with the result that ties the previous theorem into the special case $a = 0$.

\begin{prop}[{\cite[Prop.~4.8]{FHIJMTZ}}]
\label{Prop: Smooth}
	Let $k = \QQ$. For every $N \geq 1$, the curve $\pc{N,0}$ is nonsingular.
\end{prop}


\section{A Projective Embedding of $\cpc{N,a}$}\label{Sec: Alternate}

The goal of this section is to describe a closed immersion of $\pc{N,a}$ into affine $N$-space whose image has a projective closure that is a complete intersection of quadrics with no singularities on the hyperplane at infinity. It follows that if $\pc{N,a}$ is nonsingular, then the projective closure is isomorphic to $\cpc{N,a}$. We can use this embedding to gain an explicit description of the \textbf{points at infinity} --- the points of $\cpc{N,a} \smallsetminus \pc{N,a}$.

The results in this section are geometric in nature (and not arithmetic), so we will work over an arbitrary field $k$ of characteristic different from $2$. For $a\in k$, define a morphism $\psi\colon \pc{N,a} \to \Aff^N$ by
    \begin{equation*}
        \psi(x,c) = \left(x, f_c(x), f_c^2(x), f_c^3(x), \ldots, f_c^{N-1}(x) \right).
    \end{equation*}

    \begin{lem}\label{Lem: Closed im}
        The morphism $\psi\colon \pc{N,a} \to \Aff^N$ is a closed immersion. If $\Aff^N$ has coordinates $z_0, \ldots, z_{N-1}$, then the ideal defining the image of $\psi$ is
        \begin{equation*}
            I = \left( z_{N-1}^2 + z_i - z_{i-1}^2 - a : i = 1, 2, \ldots, N-1 \right).
        \end{equation*}
    \end{lem}

    \begin{proof}
        It suffices to prove that the induced homomorphism on rings of regular functions
        \begin{eqnarray*}
            \psi^*\colon\, k[z_0, \ldots, z_{N-1}] &\to& k[x,c] / (f_c^N(x) - a ) \\
            	z_i & \mapsto & f_c^i(x)
        \end{eqnarray*}
        is surjective with kernel $I$. Since $\psi^* (z_0) = x$ and $\psi^*(z_1 - z_0^2) = c$, we see that $\psi^*$ is surjective.  When $N=1$ we have $k[x,c] / (x^2 + c - a ) \cong k[x]$, showing that $\psi$ is an isomorphism and thus $I$ is trivial.  To exhibit the kernel of $\psi^*$ for $N \geq 2$, we give an explicit elimination calculation. First, we choose equivalent generators of the ideal $I$.  For $N \geq 2$, we keep the first generator of $I$ and replace the $i\tth$ generator by its difference with the first:
        \begin{equation*}
            I = \left(z_{N-1}^2 + z_1 - z_0^2 - a \right) + \left( z_i - z_{i-1}^2 - z_1 + z_0^2: i = 2, 3, \ldots, N-1\right).
        \end{equation*}
        The key fact that allows us to simplify inductively is that $f^{j-1}_c(x)^2 + c = f^j_c(x)$ for each $j \geq 1$. Now for each $j = 1, 2, \ldots, N-1$, define an ideal $I_j$ of $k[x,c, z_j, \ldots, z_{N-1}]$ by
		\begin{equation*}
			I_j = \left(z_{N-1}^2 + c - a, z_j - f_c^j(x)\right) + \left(z_i - z_{i-1}^2 - c : i = j+1, \ldots, N-1\right).
		\end{equation*}
        The map
		\begin{align*}
				k[z_0, \ldots, &z_{N-1}] / I  \to  k[x,c, z_1, \ldots, z_{N-1}] / I_1 \\
					z_0 &\mapsto x \quad \text{and} \quad z_i \mapsto z_i, \quad i \geq 1
		\end{align*}
        is an isomorphism. Indeed, the inverse of the map is given by sending $c$ to $z_1 - z_0^2$. Now we proceed by a chain of $k$-algebra isomorphisms given by eliminating one of the $z_i$'s at each step. We have
    	\begin{equation*}
    		\begin{aligned}
    			k[z_0, \ldots, z_{N-1}] / I  & \cong   k[x,c, z_1, \ldots, z_{N-1}] / I_1 \\
    				& \cong k[x, c, z_2, \ldots, z_{N-1}] / I_2 \\
    				& \  \vdots \\
    				& \cong  k[x, c, z_{N-1}] / I_{N-1} \\
    				& \cong  k[x,c] / (f_c^{N-1}(x)^2 +c - a) \\
    				& \cong  k[x,c] / (f_c^N(x) - a).
    		\end{aligned}
    	\end{equation*}
        It follows from the definitions that the map that induces the isomorphism from the first algebra to the last algebra is precisely $\psi^*$.
    \end{proof}

    The next proposition describes the projective closure of the image of $\psi$. If $\PP^N$ has homogeneous coordinates $Z_0, \ldots, Z_N$, let us identify $\Aff^N$ with the subset of $\PP^N$ where $Z_N \not= 0$.

\begin{prop}\label{affine.comp}
		\mbox{}
       	\begin{enumerate}
		[\textup(a\textup)]
    		\item\label{compint} The projective closure of the image of $\psi$ is a complete intersection of quadrics  with homogenous ideal
				\begin{equation*}
					J = \left(Z_{N-1}^2 + Z_iZ_N - Z_{i-1}^2 - aZ_N^2 : i = 1, 2, 3, \ldots, N-1 \right).
				\end{equation*}
    			    		
    		\item\label{kvalpt} The points of $V(J)$ on the hyperplane $Z_N = 0$ have homogeneous coordinates
    		    \begin{equation*}
    				\left( \epsilon_0 : \cdots : \epsilon_{N-1} : 0 \right), \qquad \epsilon_i = \pm 1.
                \end{equation*}
In particular, there are $2^{N-1}$ of them. Moreover, they are all nonsingular points of $V(J)$.
    			
    		\item\label{2^N-1pts} If $\pc{N,a}$ is nonsingular, then $\cpc{N,a} \cong V(J)$ and the complement of the affine part $\cpc{N,a} \smallsetminus  \pc{N,a}$	consists of $2^{N-1}$ points.
    	\end{enumerate}
\end{prop}

    \begin{proof}
        Let $J$ be the ideal defined in the statement of part~\eqref{compint}. We will begin by working out various properties of the scheme $V(J)$, and then we will prove that it is actually the projective closure of the image of $\psi$. Let's calculate the set-theoretic intersection of $V(J)$ with the hyperplane $Z_N = 0$. Killing $Z_N$ in all of the generators of $J$ gives the system
    	\begin{equation*}
    		Z_0^2 = Z_1^2 = \cdots = Z_{N-1}^2.
    	\end{equation*}
        As one of the coordinates must be nonzero, we may assume that all of these squares are equal to~1, and consequently all of the coordinates must be~$\pm1$. Thus we obtain the set of $k$-valued points described in the statement of~\eqref{kvalpt}. After scaling the coordinates so that the first equals~$1$, we find the other $N-1$ nonzero coordinates may be either of $\pm1$. This proves there are $2^{N-1}$ points on $V(J)$ with $Z_N= 0$.

        Next we claim that $V(J)$ has pure dimension~$1$. Indeed, it is an intersection of $N-1$ hypersurfaces, so each irreducible component of $V(J)$ must have dimension at least~$1$.
Let $W$ be the projective closure of $\im(\psi)$. If we homogenize the generators of the ideal $I$ in the previous lemma, we get exactly the ideal $J$. This shows $W \subset V(J)$ and the two schemes agree on their intersection with $\Aff^N = \{Z_N \not= 0\}$. Since $W$ has dimension~$1$, and since $V(J)$ has only finitely many points outside $\Aff^N$, we find $V(J)$ has dimension at most~$1$. We are now forced to conclude that $V(J)$ has pure dimension~$1$.  Moreover, we may infer that $W = V(J)$, as otherwise $V(J)$ would have a component of dimension zero. We have completed the proof of~\eqref{compint}.

        To see that all of the points of $V(J)$ on the hyperplane $Z_N = 0$ are nonsingular, we apply the Jacobian criterion. Dehomogenizing with respect to $Z_0$ and labeling the affine coordinates $(z_1,\ldots,z_{N})$ shows the algebraic set $V(J) \cap \{Z_0 \not= 0\}$ is cut out by the polynomials
    	\begin{eqnarray*}
    		g_1 &=&  z_{N-1}^2 + z_1z_N -1 - az_N^2  \\
    		g_2 &=& z_{N-1}^2 + z_2z_N -z_1^2 - az_N^2 \\
    		g_3 &=& z_{N-1}^2 + z_3z_N -z_2^2 - az_N^2 \\
    			& \vdots & \\
    		g_{N-1} &=& z_{N-1}^2 + z_{N-1}z_N -z_{N-2}^2 - az_N^2.
    	\end{eqnarray*}
Notice that when $z_N = 0$, none of the other $z_j$'s vanish.  The Jacobian matrix evaluated at $z_N = 0$ is given by
    	\begin{equation*}
    		\left[ \frac{\partial g_i}{\partial z_j}\right]_{z_N=0}
    			= \begin{bmatrix}
    				0 & 0 & 0 & \cdots &0 & 2z_{N-1} & z_1 \\
    				-2z_1 & 0 & 0 & \cdots & 0 & 2z_{N-1} & z_2 \\
    				0 & -2z_2 & 0 & \cdots & 0 & 2z_{N-1} & z_3 \\
    				0 & 0 & -2z_3 & \cdots & 0  & 2z_{N-1} & z_4 \\
    				\vdots & \vdots & \vdots & \ddots & \vdots & \vdots & \vdots \\
    				0 & 0 & 0 & \cdots & -2z_{N-2} & 2z_{N-1} & z_{N-1}
    			\end{bmatrix}.
    	\end{equation*}
        If we expand by minors across the first row, we can see that the determinant of the left $(N-1) \times (N-1)$ matrix is
    	\begin{equation*}
    		\pm 2^{N-1}z_1 z_2 \cdots z_{N-1} \not= 0.
    	\end{equation*}
We conclude that the Jacobian has rank $N-1$. The Jacobian criterion implies that $V(J)$ is nonsingular at each of these points. (Notice that we just used the fact that $V(J)$ has pure dimension~1.) This completes the proof of~\eqref{kvalpt}.

        Finally, we assume that $\pc{N,a}$ is nonsingular. As $\psi$ is a closed immersion, its image in $\Aff^N$ is also nonsingular. Said another way, the part of $V(J)$ outside the hyperplane $Z_N \not= 0$ is nonsingular. By part~\eqref{kvalpt}, $V(J)$ is also nonsingular at the points lying on this same hyperplane, and so we find $V(J)$ is a nonsingular complete curve birational to $\cpc{N,a}$. Up to isomorphism there is only one such curve, and so $\cpc{N,a} \cong V(J)$. This proves~\eqref{2^N-1pts}.
    \end{proof}

\begin{cor} \label{Other fields}
Fix $a \in k$. If $\pc{N,a}$ is nonsingular, then the complete curve $\cpc{N,a}$ can be defined over $k$, and it is both nonsingular and geometrically irreducible.
\end{cor}

\begin{proof}
	First, note that $\pc{N,a}$ is defined over $k$ because its defining equation has coefficients in $\ZZ[a]$.
	By Proposition~\ref{affine.comp}(\ref{compint}), we see that $\cpc{N,a}$ is defined over
	$k$; indeed, it is cut out by a collection of polynomials with coefficients in $k$. Now $\cpc{N,a}$
	is nonsingular by part (\ref{kvalpt}) of the same proposition, and it is geometrically irreducible by
	Theorem~\ref{Thm: Irreducible Genus}.
\end{proof}


\section{The Arithmetic of Rational $2^{\mathrm{nd}}$ and $3^{\mathrm{rd}}$ Pre-Images}
\label{Sec: Second and Third Pre-images}

In this section we prove Theorem~\ref{Thm: Kappa bar bound}, which states that
	\[
		\bar{\kappa}(0) = \limsup_{c \in \QQ}  \#  \Bigg\{\bigcup_{N \geq 1} f_c^{-N}(0)(\QQ) \Bigg\}= 6.
	\]

Note that $f_c^2$ has degree~4, so that $f_c$ admits at most four rational second pre-images of the origin. We can explicitly determine the $c$-values for which this happens.

\begin{prop} \label{Prop: Four Second Pre-Images}
	The rational $c$-values for which $f_c$ admits four distinct rational second pre-images of the origin are parameterized by
	\[
		c = - \frac{(t^2+ 1)^4}{16t^2(t^2-1)^2}, \quad t \in \QQ \smallsetminus \{0, \pm 1\}.
	\]
\end{prop}

\begin{proof}
	Let $c \in \QQ$ be such that $f_c$ has four rational second pre-images of~0. Then in particular, it must have two rational first pre-images, which is to say that the equation $x^2 + c = 0$ has two distinct rational solutions in~$x$. Evidently this can only happen if $c = -d^2$ for some nonzero rational number~$d$.
	
	The second pre-images are then the solutions to
		\[
			(x^2 + c)^2 + c = (x^2 - d^2)^2 - d^2 = 0 \quad \Longleftrightarrow \quad x = \pm\sqrt{d^2 \pm d}.
		\]
In order that all four of these solutions be rational, there must be $r, s \in \QQ$ such that
	\begin{equation} \label{Eq: arithmetic progression}
		r^2 = d^2 - d, \qquad s^2 = d^2 + d.
	\end{equation}
Adding these equations and dividing by $d^2$ gives
	\[
		\left(\frac{r}{d}\right)^2 +		\left(\frac{s}{d}\right)^2 = 2.
	\]
The rational solutions to the equation $a^2 + b^2 = 2$ can be parameterized by passing a line of slope~$t$ through the point $(a,b) = (-1,1)$ and determining the other point of intersection with this circle. That is, set $b = t(a+1) + 1$ and solve:
	\[
		a = \frac{r}{d} = \frac{-t^2 - 2t + 1}{t^2 + 1}, \qquad b = \frac{s}{d} = \frac{- t^2+2t + 1}{t^2 + 1}.
	\]
Divide both equations in \eqref{Eq: arithmetic progression} by $d^2$, substitute the parameterizations of $r / d$ and $s / d$ given above, and solve for $d$ to arrive at
	\[
		d = - \frac{(t^2+1)^2}{4t(t^2 - 1)} \quad
			\Longrightarrow \quad  c = -d^2 = - \frac{(t^2+1)^4}{16t^2(t^2 - 1)^2}.
	\]
Evidently $t \neq 0,\pm 1$.

Conversely, a direct computation shows that for any $t \in \QQ \smallsetminus \{0, \pm 1\}$, the map $f_c$ with $c = - \frac{(t^2+ 1)^4}{16t^2(t^2-1)^2}$ has four distinct rational pre-images of the origin, namely
	\begin{equation} \label{Eq: Second Pre-Images}
		x = \pm \frac{(t^2 + 1)(t^2 \pm 2t - 1)}{4t(t^2 - 1)}.
	\end{equation}

\end{proof}

\begin{lem} \label{Lem: Elliptic Curve}
	The complete curve $\cpc{3,0}$ is $\QQ$-isomorphic to the elliptic curve
	$E$ with affine equation  $v^2 = u^3 - u + 1$. The birational transformation $E \dashrightarrow \pc{3,0}$
	 is given by
	\begin{equation} \label{Eq: Birational}
		x =  \frac{v}{u^2 - 1} \qquad c =   \frac{-1}{(u^2-1)^2}.
	\end{equation}
	Moreover, this elliptic curve has Mordell-Weil group $E(\QQ) \cong \ZZ$.
\end{lem}

	
\begin{proof}
	We begin by working with the nonsingular projective model $\cpc{3,0}$ for $\pc{3,0}$.
	Proposition~\ref{affine.comp}\eqref{compint} shows $\cpc{3,0}$ can be embedded in $\PP^3$
	 as the intersection of two quadric hypersurfaces:
	\begin{equation*}
			Z_0^2 = Z_2^2 + Z_1Z_3   \qquad  Z_1^2= Z_2^2 + Z_2Z_3.
	\end{equation*}
Setting $Z_2 = 0$ shows $(0: 0: 0: 1)$ is the unique point on $\cpc{3,0}$ lying on the hyperplane $Z_2 = 0$.
We now pass to affine coordinates that place this point at infinity by setting $v = Z_0 / Z_2$, $u = Z_1 / Z_2$, and $w = Z_3 / Z_2$.  The above equations then become
	\[
		v^2 = 1 + uw  \qquad u^2 =  1 + w.
	\]
Eliminating the variable $w$ yields $v^2 = u^3 - u + 1$, which proves that $\cpc{3,0}$ is isomorphic to the desired elliptic curve.

Tracing through the definitions of the maps $\pc{3,0} \hookrightarrow \Aff^3 \hookrightarrow \PP^3$ shows that the birational transformation $\alpha: \pc{3,0} \dashrightarrow E$ that we have just  constructed is given by
	\begin{equation} \label{Eq: u,v}
		u = \frac{f_c(x)}{f_c^2(x)}  \qquad v = \frac{x}{f_c^2(x)}.
	\end{equation}
If we define $\beta: E \dashrightarrow \pc{3,0}$ by the formulas in \eqref{Eq: Birational}, then a straight-forward (albeit messy) computation shows $\alpha \circ \beta = \mathrm{id}$ and $\beta \circ \alpha = \mathrm{id}$.

Cremona's elliptic curve database \cite{Cremona_Data} tells us the curve $E$ has rank~1 and no torsion.
\end{proof}

The following corollary is an immediate consequence of Lemma \ref{Lem: Elliptic Curve}.

\begin{cor}
	The $c$-values for which $f_c$ admits a third rational pre-image of the origin are given by $c=0$ and $c = -(u^2 - 1)^{-2}$, where $(u,v)$ is an affine rational point on the elliptic curve $v^2 = u^3 - u + 1$ and $u \neq \pm 1$.
\end{cor}

\begin{prop} \label{Prop: Elliptic Curve}
	Given $c_0 \in \QQ$, there are at most two rational third pre-images of the origin for the morphism $x \mapsto f_{c_0}(x)$.
\end{prop}

\begin{proof}
	For the proof, it will be useful to distinguish between the curve $\cpc{3,0}$ and the elliptic curve $E$ with affine Weierstrass model $v^2 = u^3 - u + 1$, even though the preceding lemma shows that they are isomorphic over $\QQ$. For a point $P \in E(\QQ)$, we write $u(P)$ for its $u$-coordinate. (Here $u(P) = \infty$ if $P$ is the origin for the group law.)
	
	Let $h$ be the naive logarithmic height on $\PP^1$. The function $g(P) = -(u(P)^2 - 1)^{-2}$ is an even rational function on $E$, and so we may define a height function by $h_g(P) = h(g(P))$. The strategy of the proof is to show that if $(x_1, c_0)$ and $(x_2, c_0)$ are two points on $\cpc{3,0}(\QQ)$ corresponding to rational third pre-images of the origin, then they in turn correspond to two points $P_1$ and $P_2$ on the elliptic curve $E$. The function $g$ cannot distinguish between points with the same $c$-coordinate, and so $h_g(P_1) = h_g(P_2)$. On the other hand, since $E$ has rank~1, we will be able to show that if a point $P$ has sufficiently large height, then $-P$ is the only other point with the same height. This will reduce the problem to a finite amount of computation. We now make this strategy more explicit, although we omit many of the computational details.
	
	Let $\hat{h}$ be the canonical height on $E$ (with respect to the divisor $(\infty)$). We can bound the difference between $8\hat{h}$ and $h_g$ as follows. For ease of notation, write $h_u(P) = h(u(P))$. First, it is easily seen that $|h(u^2-1) - 2h(u)| \le \log 2$. This implies that
	\[ \left|4 h_u(P) - h_g(P)\right| \le 2 \log 2 \]
	for all $P \in E(\QQ)$. In the paper~\cite{CrePriSik}, it is shown how an upper bound for $h_u - 2 \hat{h}$ can be computed. The corresponding functionality is available in \textit{Magma} as \texttt{SiksekBound}. We find a bound of $< 0.47$. On the other hand, we obtain from the explicit form of the duplication map on the $u$ coordinate
	\[ u(2P) = \frac{u(P)^4 + 2 u(P)^2 - 8 u(P) + 1}{4(u(P)^3 - u(P) + 1)} \]
	the trivial bound $h_u(2P) \le 4 h_u(P) + \log 12$. This implies a lower bound of $-(\log 12)/3 > -0.83$ for $h_u - 2 \hat{h}$. Combining these estimates with the equality
	\[ 8 \hat{h}(P) - h_g(P)
	     = 4\bigl(2\hat{h}(P) - h_u(P)\bigr) + \bigl(4 h_u(P) - h_g(P)\bigr) ,
  \]
  we have 
	\[ -3.27 \le -4 \cdot 0.47 - 2 \log 2 \le 8 \hat{h}(P) - h_g(P)
           \le 4 \cdot 0.83 + 2 \log 2 \le 4.71
  \]
	for all $P \in E(\QQ)$.
	
	As $E(\QQ)$ has rank~1, we may choose a generator $P_0$. For any $n \geq 1$, the above estimate and properties of the canonical height show
	\begin{equation} \label{Eq: height difference}
          \begin{aligned}
            h_g\left( [n+1]P_0 \right) - h_g \left( [n]P_0 \right)
              &\ge 8\hat{h}\left([n+1]P_0\right) - 8\hat{h}\left([n]P_0\right) - 7.98 \\
              &= 8(n+1)^2 \ \hat{h}(P_0) - 8n^2 \ \hat{h}(P_0) - 7.98 \\
              &= 8(2n+1) \ \hat{h}(P_0)  - 7.98.
          \end{aligned}
	\end{equation}
	The point $P_0 = (1,1)$ is a generator of $E(\QQ)$, and its canonical height is $\hat{h}(P_0) \approx 0.0249$ according to  \textit{PARI/gp} or \textit{Magma}.\footnote{Warning: \textit{PARI/gp} and \textit{Magma} compute the canonical height with respect to the divisor $2(\infty)$ on $E$. The canonical height of $P_0$ is given here with respect to the divisor $(\infty)$.} It follows that the final quantity in \eqref{Eq: height difference} is positive as soon as
	\[
		n \geq \frac{1}{2}\left( \frac{7.98}{8 \hat{h}(P_0)} -1 \right)
		  \approx 19.53 \,.
	\]

	To conclude the proof, suppose that $c_0 \in \QQ$ and $x_1, x_2$ are two rational third pre-images of the origin for the map $x \mapsto x^2 + c_0$. We aim to show that $x_1 = \pm x_2$. To that end, we may assume that $c_0 \not= 0$ since for $c_0 = 0$ there is exactly one third pre-image of the origin. The pre-images $x_1$ and $x_2$ correspond to two points $P_1$ and $P_2$ in $E(\QQ)$ such that $g(P_1) = g(P_2) = c_0$. The fact that $c_0 \not= 0$ implies that neither $P_i$ is the origin for the group law on $E$. Recalling that $P_0$ is our fixed generator for $E(\QQ)$, there exist nonzero integers $n_1, n_2$ such that $P_i = [n_i]P_0$. Moreover, replacing $P_i$ by $-P_i$ has the effect of replacing $x_i$ by $-x_i$, as can readily be seen from \eqref{Eq: Birational}. Hence we may assume that $n_i > 0$. After reordering if necessary, we may further suppose that $n_1 \geq n_2$. If they are equal, then we are finished, so assume $n_1 > n_2$.
	
	If $n_2 \geq 20$, then the computation in \eqref{Eq: height difference} implies $h_g(P_1) > h_g(P_2)$. But $g(P_1) = g(P_2)$, so this is a contradiction. Hence $1 \leq n_2 \leq 19$. In fact, $n_2 \ge 3$, since $P_0$ and $2P_0$ lead to an infinite value of~$c$. For this range of $n$, we now compute $g([n]P_0)$ --- the $c$-coordinate for the point corresponding to $[n]P_0$ --- and verify that these values are pairwise distinct. It follows that $n_1 > 19 \geq n_2 \geq 3$. Finally, we verify that $h_g([20]P_0) > h_g([n]P_0)$ for all $n < 20$, so that $h_g(P_1) > h_g(P_2)$ by \eqref{Eq: height difference} again. This contradiction completes the proof.
\end{proof}

\begin{prop} \label{Prop: 2-4-2}
	There are no values $c \in \QQ$ such that the map $f_c$ admits four distinct rational second pre-images and a rational third pre-image of the origin.
\end{prop}

\begin{proof}
	We reduce the proof to the determination of the set of rational points on a certain hyperelliptic curve of genus~$3$.

	Let $c_0 \in \QQ$ be such that the map $f_{c_0}$ admits four distinct rational second pre-images and a rational third pre-image of the origin. Proposition~\ref{Prop: Four Second Pre-Images} shows that there exists $t \in \QQ \smallsetminus \{0, \pm 1\}$ such that
	\[
		c_0 = 	- \frac{(t^2 + 1)^4}{16t^2(t^2 - 1)^2}.
	\]
Let $x_0$ be a rational third pre-image of the origin for $f_{c_0}$. Then $f_{c_0}(x_0)$ is a rational second pre-image of the origin, and so it is given by \eqref{Eq: Second Pre-Images}. Also, $f_{c_0}^2(x_0) = \pm \sqrt{-c_0}$ is a rational first pre-image of the origin. Since $f_{c_0}$ admits a rational third pre-image of the origin, Lemma~\ref{Lem: Elliptic Curve} shows there is a rational pair $(u,v)$ satisfying $v^2 = u^3 - u + 1$, and equation~\eqref{Eq: u,v} gives a formula for $u$:
	\[
		u = \frac{f_{c_0}(x_0)}{f_{c_0}^2(x_0)} = \eps_1 \frac{t^2 + 2 \eps_2 t - 1}{t^2 +1}
	\]
with $\eps_1, \eps_2 = \pm 1$. 
Thus the equation $v^2 = u^3 - u + 1$ becomes
	\begin{align*}
		v^2 &= \eps_1 \left(\frac{t^2 + 2 \eps_2 t - 1}{t^2 +1}\right)^3 
             - \eps_1 \frac{t^2 + 2\eps_2 t - 1}{t^2 +1} + 1 \\
		&= \frac{(t^2+1)(t^6 + 4 \eps_1 \eps_2 t^5 + (3 + 8 \eps_1)t^4 - 8 \eps_1 \eps_2 t^3
               + (3 - 8 \eps_1) t^2 + 4 \eps_1 \eps_2 t + 1)}{(t^2 + 1)^4}.
	\end{align*}

  Define a set of four smooth projective curves~$C_{\eps_1,\eps_2}$ over~$\QQ$ as the 
  smooth projective models associated to the affine curves given by
  \[ y^2 = (t^2+1)(t^6 + 4 \eps_1 \eps_2 t^5 + (3 + 8 \eps_1) t^4
                    - 8 \eps_1 \eps_2 t^3 + (3 - 8 \eps_1) t^2 + 4 \eps_1 \eps_2 t + 1) 
                    . \]
  These curves are pairwise isomorphic via $(t,y) \mapsto (\pm t^{-1},yt^{-4})$
  or $(t,y) \mapsto (-t,y)$.
  
  Given $c_0$, there are $t,u,v \in \QQ$ satisfying the relations above, and, setting
  $y = v (t^2+1)^2$, we obtain a rational point on $C_{\eps_1,\eps_2}$ 
  (for some choice of $\eps_1, \eps_2 = \pm 1$) satisfying
  $t \notin \{0, 1, -1, \infty\}$. To prove the proposition, it is therefore
  sufficient to show that no such rational point exists on any of the curves.
  Since the isomorphisms between the curves identify the sets of such rational
  points on them, we may restrict our attention to $C = C_{+1,+1}$, say.
  We do find four pairs of rational
  points on~$C$ with $t$-coordinate in the `forbidden set' $\{0,1,-1,\infty\}$,
  so our task is to show that these eight points exhaust the rational points
  on~$C$.
  
  We observe that $C$ has two `extra automorphisms' (i.e., nontrivial automorphisms
  distinct from the hyperelliptic involution) given by the involutions
  \[ (t,y) \longmapsto \Bigl(\frac{t+1}{t-1}, \frac{4y}{(t-1)^4}\Bigr)
     \quad\text{and}\quad
     (t,y) \longmapsto \Bigl(\frac{t+1}{t-1}, -\frac{4y}{(t-1)^4}\Bigr) .
  \]
  The first one of these has four fixed points (with $t = 1 \pm \sqrt{2}$), the
  second one is without fixed points. Therefore the quotient of~$C$ by the
  first involution is an elliptic curve, whereas the quotient by the second
  involution is a curve of genus~$2$. The elliptic curve is the curve describing
  rational third pre-images of the origin. The genus~$2$ quotient~$D$ is the
  smooth projective curve associated to the affine equation
  \[ w^2 = z^5 - 3 z^3 - z^2 + 2 z + 2 = (z^2 - 2)(z^3 - z - 1) . \]
  The map from $C$ to~$D$ is given by
  \begin{equation} \label{mapCtoD}
     (t, y) \longmapsto \Bigl(\frac{-t^2-2t+1}{t^2+1}, 
                              \frac{-t^2+2t+1}{(t^2+1)^3}\,y\Bigr) .
  \end{equation}
  A 2-descent on the Jacobian~$J_D$ of~$D$ as described in~\cite{Stoll2Descent} results
  in an upper bound of~1 for the rank of~$J_D(\QQ)$. The torsion subgroup has order~2.
  Using a bound for the difference between the naive and canonical heights on~$J_D$, we determine
  that $P = [(1,1) - (-1,-1)]$ generates the free part of~$J_D(\QQ)$; 
  compare~\cite{StollHeights1,StollHeights2}. We thus have determined~$J_D(\QQ)$.
  Since the rank is less than the genus, we can apply Chabauty's method;
  see for example~\cite{Stoll_Independence_2006}. Since 2 and~23 are the only primes of
  bad reduction for~$D$, we will work at~$p = 3$.
  
  The image of~$P$ in~$J_D(\FF_3)$ has order~9. Computing $9P$, which is in the
  kernel of reduction, we find that a differential that kills the Mordell-Weil group
  reduces mod~3 to $\omega = z\,dz/w$. Since $\omega$ does not vanish at any
  point in~$D(\FF_3)$, this implies by \cite[Prop.~6.3]{Stoll_Independence_2006} that
  the map $D(\QQ) \to D(\FF_3)$ is injective. On the other hand, we have $\#D(\FF_3) = 5$,
  and we can identify five rational points on~$D$ (with $z = -1, 1, \infty$).
  Therefore these points must be all of the rational points on~$D$.
  Looking at the fibers of these five points under the map given in~\eqref{mapCtoD},
  we see that $C$ cannot have rational points other than those we already know.
  (The points on~$D$ with $z = \pm 1$ each give rise to two rational points on~$C$,
  whereas the points lying above the point at infinity on~$D$ are not rational.)
  
  Note also that \textit{Magma} contains an implementation of a method described
  in~\cite[Sect.~4.4]{BruinStoll} that determines the set of rational points
  on a curve of genus~2 whose Jacobian has Mordell-Weil rank~1, when a generator
  of the free part of the Mordell-Weil group is known. It can be applied
  to the curve~$D$ and the point~$P$ above.
\end{proof}

\begin{proof}[Proof of Theorem~\ref{Thm: Kappa bar bound} and Theorem~\ref{Thm: Kappa bound}]
	Proposition~\ref{Prop: Four Second Pre-Images} exhibits an infinite family of rational parameters~$c$ such that there are four distinct rational second pre-images of the origin. Applying $f_c$ to these second pre-images gives two distinct rational first pre-images. So for each of these parameters $c$, we have
		\[
			\kappa(0) \geq \bar{\kappa}(0) \geq \#f_c^{-1}(0)(\QQ) + \#f_c^{-2}(0)(\QQ) = 6.
		\]
		
Let $S$ be the set of $c \in \QQ$ such that $f_c$ admits a rational $4\tth$ pre-image of the origin. The affine curve $\pc{4,0}$ has genus~5 (Theorem~\ref{Thm: Irreducible Genus}), and so by Faltings' theorem the set $S$ is finite.

Now suppose that $c$ is a rational parameter in the complement of the finite set $S$. Then $f_c^{-4}(0)(\QQ)$ is empty. If $f_c^{-3}(0)(\QQ)$ is empty, then looking at the degrees of $f_c$ and $f_c^2$ shows that $f_c$ admits at most six rational iterated pre-images of the origin. On the other hand, if $f_c^{-3}(0)(\QQ)$ is nonempty, then Proposition~\ref{Prop: Elliptic Curve} implies that it contains at most two elements.
By Proposition~\ref{Prop: 2-4-2}, it follows that $f_c$ has at most two distinct rational second pre-images of the origin. There are never more than two first pre-images, so $f_c$ has at most six rational iterated pre-images again. Hence $\bar{\kappa}(0) \leq 6$, which completes the proof of Theorem~\ref{Thm: Kappa bar bound}.

	The hypothesis of Theorem~\ref{Thm: Kappa bound} is that $S = \{0, -1\}$. A direct calculation shows
	\[
		\#\left\{\bigcup_{N \geq 1} f^{-N}_0(0)(\QQ) \right\} = 1 \quad \text{and} \quad
		\#\left\{\bigcup_{N \geq 1} f^{-N}_{-1}(0)(\QQ) \right\}  = 3.
	\]
When combined with the above arguments, we see $\kappa(0) = 6$. 
\end{proof}

\begin{remark}
	Using techniques similar to the ones given here, the second author and two of his undergraduate students have performed a detailed analysis of $\bar{\kappa}(a)$ for $a \in \bar{\QQ}$. It turns out that $\bar{\kappa}(a) = 6$ for all $a \in \QQ \smallsetminus \{-1/4\}$, and $\bar{\kappa}(-1/4) = 10$. Upon suitably extending the definition of $\bar{\kappa}(a)$ to an arbitrary number field containing $a$, this result continues to hold for any $a \in \bar{\QQ}$ outside of an explicit finite set. The main difference from the present work lies in the fact that $\cpc{3,a}$ has generic rank~2, and so the method of Proposition~\ref{Prop: Elliptic Curve} does not carry over. However, it suffices to show that the possible arrangements of more than six pre-images (or ten in the case $a = -1/4$) correspond to a finite collection of algebraic curves of genus at least two. See \cite{Hutz_et_al} for  details.
\end{remark}


\section{The Arithmetic of $4\tth$ Pre-images}
\label{Sec: Arithmetic}


Recall from the introduction that we defined $\kappa(0)$ to be the maximum number of rational iterated pre-images for a morphism $f_c(x) = x^2 + c$:
	\[
		\kappa(0) = \sup_{c \in \QQ} \#\left\{ \bigcup_{N \geq 1} f^{-N}_c(0)(\QQ)\right\}.
	\]
Controlling $\kappa(0)$ hinges on the finiteness of the rational points of the curve $\pc{4,0}$ or, equivalently, the pre-image curve $\cpc{4,0}$. A height theoretic argument was necessary to connect this result to the higher pre-images $f_c^{-N}(0)(\QQ)$ for $N > 4$ \cite{FHIJMTZ}.  It turns out that a strong bound for $\#\cpc{4,0}(\QQ)$ allows us to bound $\kappa(0)$ without recourse to height machinery.

\begin{thm}
\label{Thm: kappa is 10}
	If $\#\cpc{4,0}(\QQ) = 10$, then $\kappa(0) = 6$.
\end{thm}

Before proving Theorem~\ref{Thm: kappa is 10}, we need to collect a few facts about the complete curve $\cpc{4,0}$.
By Theorem~\ref{Thm: Irreducible Genus} and Proposition~\ref{Prop: Smooth}, we see that $\cpc{4,0}$ is a nonsingular, geometrically irreducible curve of genus~$5$. Using Proposition~\ref{affine.comp}, we can embed the fourth pre-image curve in projective space as
	\begin{equation}
	\label{Eq: X(4,0) embedding}
        	\cpc{4,0} \cong V(Z_3^2 + Z_1Z_4 -Z_0^2, \ Z_3^2+Z_2Z_4-Z_1^2, \ Z_3^2+Z_3Z_4-Z_2^2) \subset \PP^4.
	\end{equation}
Note that this is the canonical embedding of $\cpc{4.0}$.

	Our curve has ten rational points that one locates easily by inspection. There are the eight points at infinity given by Proposition~\ref{affine.comp}(\ref{kvalpt}, \ref{2^N-1pts}):
    \begin{align*}
        P_1 &= (1:1:1:1:0) &  P_2 &= (1:1:1:-1:0)\\
        P_3 &= (1:1:-1:1:0) & P_4 &= (1:1:-1:-1:0)\\
        P_5 &= (1:-1:1:1:0) & P_6 &= (1:-1:1:-1:0)\\
        P_7 &= (1:-1:-1:1:0) & P_8 &= (1:-1:-1:-1:0).
\intertext{The origin in $\Aff^1(\QQ)$ is periodic for the morphisms $f_{0}$ and $f_{-1}$, with periods $1$ and $2$, respectively. Using the embedding in Lemma~\ref{Lem: Closed im}, these correspond to points on $\cpc{4,0}$:}
        P_9 &= (0:0:0:0:1) & P_{10} &= (0:-1:0:-1:1).
    \end{align*}

	Only $P_9$ and $P_{10}$ correspond to parameters $c \in \QQ$, and hence Theorem~\ref{Thm: kappa is 10} is a reformulation of Theorem~\ref{Thm: Kappa bound} from the introduction.



\medskip

We now turn to the task of bounding the number of rational $4\tth$ pre-images of the origin --- or what amounts to the same thing --- bounding the size of $\cpc{4,0}(\QQ)$.  Much of this and the next section is based on \cite{Stoll3}.  
 For many of our calculations we used the computer algebra system \textit{Magma} \cite{magma}. 

The curve $\cpc{4,0}$ can be embedded in $\PP^4$ as in \eqref{Eq: X(4,0) embedding}, and the image is defined over $\ZZ$. \textit{Magma} has determined that $\cpc{4,0}$ has good reduction outside the primes $2$, $23$, and $2551$. 
We will use several primes of good reduction in the arguments below.
	
We first provide a number of facts about the Jacobian~$\jpc{4,0}$ of the curve~$\cpc{4,0}$.

\begin{thm} \label{thm_jac_properties}
  \strut
  \begin{enumerate}[\textup(a\textup)]
    \item\label{jac_prop_1}
      The Jacobian $\jpc{4,0}$ is isogenous to the product of a simple
      abelian variety of dimension $4$ and the elliptic curve~$E$ with Weierstrass equation $v^2=u^3-u+1$.
    \item\label{jac_prop_3}
      The subgroup of $\jpc{4,0}(\QQ)$ generated by the ten known rational points on $\cpc{4,0}$ is isomorphic to $\ZZ^3$. In fact, it is already generated by divisors supported on the points at infinity.
  \end{enumerate}
\end{thm}

\begin{proof}
  We can define a morphism $\delta : \cpc{4,0} \to \cpc{3,0}$ as follows. For points of the affine
  piece $\pc{4,0}$, the morphism is given by $(x,c) \mapsto (x^2+c, c)$. In other words, it sends a $4\tth$ pre-image of the origin to a $3\rrd$ pre-image. As $\cpc{4,0}$ and $\cpc{3,0}$ are complete and nonsingular, the morphism extends over $\cpc{4,0}$.
  
  We saw in the proof of Lemma~\ref{Lem: Elliptic Curve} that $\cpc{3,0}$ is an elliptic curve isomorphic to the curve~$E$ with Weierstrass equation $v^2 = u^3-u+1$. In particular, $\cpc{3,0}$ is isomorphic to its Jacobian $\jpc{3,0}$. Passing to the morphism on Jacobians induced by $\delta$
  \begin{equation*}
    \jpc{4,0} \to \jpc{3,0} \cong \cpc{3,0},
  \end{equation*}
  we see $\jpc{4,0}$ splits up to isogeny as a product of $\jpc{3,0}$ and another abelian variety of dimension~4. We must check that the larger factor is simple.

  The Weil conjectures allow us to compute the Euler factor of $\jpc{4,0}$ at $p=3$ by computing the cardinality $\#\cpc{4,0}(\FF_{3^m})$ for $m=1,\ldots,5$.
  We find the Euler factor to be
  \begin{equation*}
      (T^2 + 3T + 3)(T^8 + 3T^7 + 7T^6 + 16T^5 + 28T^4 + 48T^3 + 63T^2 + 81T + 81).
  \end{equation*}
  Since the two factors are irreducible and since the splitting of the Jacobian is witnessed by a splitting of the Euler factors, the Jacobian $\jpc{4,0}$ splits into a product of at most two simple Abelian varieties. The first polynomial is exactly the Euler factor for the elliptic curve $v^2=u^3-u+1$, completing the proof of (\ref{jac_prop_1}). 

  The prime-to-$p$ torsion in $\jpc{4,0}(\QQ)$ injects into
  $\jpc{4,0}(\FF_p)$ for any prime $p$ of good reduction \cite[Thm.~C.1.4]{Hindry_Silverman_book_2000}.
  It suffices for our purposes to compute $\jpc{4,0}(\FF_p)$ for three odd primes :
  \begin{align*}
      \#\jpc{4,0}(\FF_5) &= 2^3\cdot 877\\
      \#\jpc{4,0}(\FF_7) &= 2^4 \cdot 3\cdot 7\cdot 233\\
      \#\jpc{4,0}(\FF_{11}) &= 2^2 \cdot 5\cdot 7\cdot 13759.
  \end{align*}
  It follows that the torsion subgroup has order $1, 2$, or $4$. Since
  \begin{equation*}
    \jpc{4,0}(\FF_3) \cong \ZZ/2296\ZZ \quad\text{and}\quad
    \jpc{4,0}(\FF_{29}) \cong \ZZ/2\ZZ \oplus \ZZ/11284630\ZZ,
  \end{equation*}
  the torsion subgroup is either trivial or of order~2.

  The main tool for proving the remaining portion of the assertions is the homomorphism
  \begin{equation*}
      \Phi_S\colon \bigoplus_{i=1}^{10}\ZZ P_i \to \Pic_{\cpc{4,0}}(\QQ)
          \to \prod_{p \in S} \Pic_{\cpc{4,0}/\FF_p}(\FF_p)
  \end{equation*}
  where $S$ is a set of primes of good reduction.  Here $\Phi_S$ maps a divisor to its class in
  $\Pic_{\cpc{4,0}}(\QQ)$, and then sends it to the tuple consisting of its reductions for primes in $S$. Let $G$ be the subgroup of $\jpc{4,0}(\QQ)$ generated by the ten points $P_i$. Then $G$ is given by the image of the degree-zero part of $\bigoplus_{i=1}^{10} \ZZ P_i$ inside $\Pic_{\cpc{4,0}}(\QQ)$. \textit{A priori} we see that the rank of $G$ is at most 9.

  Take $S =\{3,5,7,11,13,17\}$ and consider the kernel of $\Phi_S$. Some elements of the kernel correspond to linear equivalence relations among the $P_i$, and others are artifacts of the reductions modulo primes of $S$.  Using \textit{Magma} we can exhibit rational functions on $\cpc{4,0}$ that show the following (independent) linear equivalence relations:
  \begin{equation} \label{Eq: Relations}
    \begin{aligned}
      P_3 + P_4 + P_7 + P_8 &\sim P_1 + P_6 + P_9 + P_{10}\\
      P_2 + P_3 + P_6 + P_7 &\sim P_1 + P_4 + P_5 + P_8\\
      P_3 + P_8 + P_9 + P_{10} &\sim P_1 + P_2 + P_5 + P_6\\
      P_5 + P_6 + P_7 + P_8 &\sim 2P_1 + P_2 + P_9\\
      P_4 + P_5 + 2P_9 &\sim P_1 + P_2 + P_7 + P_8\\
      P_1 + P_2 + P_7 + P_8 &\sim P_3 + P_6 + 2P_{10}.
    \end{aligned}
  \end{equation}
  These six relations show that $G$ has rank at most three.

  From the definitions, we can see that the map $\Phi_S$ induces a homomorphism
  \begin{equation*}
    G \to \prod_{p \in S} \Pic^0_{\cpc{4,0}/\FF_p}(\FF_p)
  \end{equation*}
  whose image is equal to the degree-zero part of the image of $\Phi_S$, denoted $\im(\Phi_S)^0$. We can compute the image of $\Phi_S$ 
by taking the quotient of $\bigoplus \ZZ P_i$ by the kernel of $\Phi_S$; passing to the degree-zero part, we find $\im(\Phi_S)^0$ is isomorphic (as an abelian group) to
  \begin{equation*}
  \label{Eq: Invariants}
      \ZZ/28\ZZ \times \ZZ/1680\ZZ \times \ZZ/392857929291088111200\ZZ.
  \end{equation*}

  As $G$ contains no $7$-torsion points while its quotient $\im(\Phi_S)^0$ contains a factor $(\ZZ/7\ZZ)^3$, we conclude that $G$ has rank at least three. Hence its rank is exactly three.

  Now let $G'$ be the degree-zero part of $\bigoplus_{i=1}^{10}\ZZ P_i$ modulo the relations in
  \eqref{Eq: Relations}. We may use these relations to eliminate the generators $P_9$, $P_{10}$, $P_1$, $P_4$, $P_7$, and then $P_6$. (It is of course possible to do the elimination in other ways.). One now sees that, as abelian groups,
  \[
     G' \cong \ZZ (P_2 - P_3) \oplus \ZZ (P_2-P_5) \oplus \ZZ (P_2-P_8).
  \]
  As $G'$ surjects onto $G$
  and as they have the same rank, we find $G$ is also free abelian. Moreover, this calculation shows that $G$ is generated by divisors supported on the points at infinity. The proof of (\ref{jac_prop_3}) is now complete.
\end{proof}
	
The following result proves Theorem~\ref{Thm: Bound for X(4,0)}.

\begin{thm} \label{thm_fourth_preimages}
  If the rank of $\jpc{4,0}(\QQ)$ is $3$, then $c=0$ and $c=-1$ are the only rational values of~$c$ such that $f_c^{-4}(0)$ is nonempty. 
\end{thm}

\begin{proof}
  By Theorem~\ref{Thm: kappa is 10},
  the statement is equivalent to $\cpc{4,0}(\QQ) = \{P_1, P_2, \dots, P_{10}\}$,
  and this is what we will prove under the assumption that the rank
  of~$\jpc{4,0}(\QQ)$ is~$3$. We will again use the Chabauty-Coleman method,
  in a similar spirit as in the proof of Proposition~\ref{Prop: 2-4-2}. We refer
  again to~\cite{Stoll_Independence_2006} for the general background and
  to~\cite{Stoll3} for another example done in a similar way.
  
  Since $p = 3$ is the smallest odd prime of good reduction, we will work
  over~$\QQ_3$. Recall that there is a pairing between $\jpc{4,0}(\QQ_3)$
  and the space of regular differentials with coefficients in~$\QQ_3$
  on~$\cpc{4,0}$ that is $\QQ_3$-linear in the second argument and becomes
  $\ZZ_3$-linear in the first argument when restricted to the kernel of
  reduction. Its kernel on the left is the torsion subgroup, and the kernel
  on the right is trivial.
  Our assumption then implies that there are two linearly independent
  differentials that annihilate the Mordell-Weil group~$\jpc{4,0}(\QQ)$.
  In order to determine them, we first find three independent points in
  the kernel of reduction. Taking $P_5$ as base-point for an embedding
  $\iota : \cpc{4,0} \to \jpc{4,0}$, we can take for example
  \begin{align*}
    Q_1 &= \iota(P_2) + \iota(P_4) + \iota(P_6) + \iota(P_7)
             + \iota(P_9) + \iota(P_{10})\,, \\
    Q_2 &= -\iota(P_2) - \iota(P_4) - \iota(P_6) + \iota(P_7)
             - \iota(P_8) + \iota(P_{10})\,, \\
    Q_3 &= 2 \iota(P_1) - \iota(P_2) + 2 \iota(P_6) - \iota(P_8) + \iota(P_{10})\,.
  \end{align*}
  We represent each~$Q_j$ by a divisor of the form $D_j - 5 P_5$ with an
  effective divisor~$D_j$ of degree~$5$ all of whose points have the same
  reduction mod~$3$ as~$P_5$. We pick a uniformizer~$t$ at~$P_5$ that
  reduces to a uniformizer at the reduction of~$P_5$ on $\cpc{4,0}/\FF_3$.
  Since we are using the canonical model of~$\cpc{4,0}$, the differentials
  correspond to hyperplane sections. We fix a differential~$\omega_0$
  that reduces to a non-zero differential mod~$3$ and
  corresponds to~$Z_0 = 0$, and then set $\omega_k = (Z_k/Z_0) \omega_0$,
  for $k = 1,2,3,4$. This gives us a basis $\omega_0, \dots, \omega_4$
  of the space of regular differentials.
  We can then write each $\omega_k$ in the form
  \[ \omega_k = f_k(t)\,dt \qquad\text{with a power series $f_k \in \QQ[\![t]\!]$.} \]
  We integrate formally to obtain power series $l_k(t)$. Writing
  $D_j = P_{j,1} + \ldots + P_{j,5}$, we obtain the value of the pairing
  between $Q_j$ and~$\omega_k$ as
  \[ \langle Q_j, \omega_k \rangle = \sum_{i=1}^5 l_k\bigl(t(P_{j,i})\bigr) \,. \]
  The power series converge at $t(P_{j,i})$ since the latter has positive
  valuation. Doing the computation to an absolute precision of~$O(3^5)$,
  we obtain the following result.
  \[ \bigl(\langle Q_j, \omega_k \rangle\bigr)_{1 \le j \le 3, 0 \le k \le 4}
      \equiv  3 \begin{pmatrix}
                    0 & -15 &  24 &  17 & -29 \\
                  -13 & -25 &  23 & -33 &  17 \\
                   13 & -40 & -34 &  17 & -12
                \end{pmatrix}
      \bmod 3^5
  \]
  We compute the kernel of the matrix and find that
  \[ \eta_1 = \omega_1 + 32 \omega_3 + 35 \omega_4 \qquad\text{and}\qquad
     \eta_2 = \omega_2 - 7 \omega_3 - 34 \omega_4
  \]
  are (up to $O(3^4)$) a basis of the space of differentials annihilating
  the Mordell-Weil group. In particular, the reductions of these differentials
  mod~$3$ correspond to the hyperplane sections
  \[ Z_1 - Z_3 - Z_4 = 0 \qquad\text{and}\qquad Z_2 - Z_3 - Z_4 = 0 \] 
  on~$\cpc{4,0}/\FF_3$.

  We determine the set $\cpc{4,0}(\FF_3)$ and observe that reduction mod~$3$
  gives a bijection between $\{P_1, P_2, \dots, P_{10}\}$ and this set.
  Writing $\bar{P}$ for the reduction mod~$3$ of a point $P \in \cpc{4,0}(\QQ)$,
  we see that at least one of the reductions of $\eta_1$ and~$\eta_2$ does
  not vanish at~$\bar{P}_j$, for all $j \in \{1,2,\dots,10\}$ except
  $j = 1$ and~$j = 8$. According to~\cite[Prop.~6.3]{Stoll_Independence_2006},
  this implies that for all $j \in \{2,3,4,5,6,7,9,10\}$, $P_j$ is the only
  rational point that reduces mod~$3$ to~$\bar{P}_j$. It remains to deal
  with the residue classes of $P_1$ and~$P_8$. Here we will exploit the
  fact that we have two independent differentials at our disposal.
  Let $P = P_1$ or~$P_8$ and let, similarly as before, $T$ denote a
  uniformizer at~$P$ that reduces to a uniformizer at~$\bar{P}$. We
  write $\eta_1$ and~$\eta_2$ as power series in~$T$ times~$dT$
  and integrate to obtain two analytic functions $\lambda_1$ and~$\lambda_2$
  on the residue class of~$P$ (parameterized by $T \in 3\ZZ_3$)
  that have to vanish at all rational points in that residue class.
  We write $T = 3\tau$ with $\tau \in \ZZ_3$ and consider $\lambda_1$
  and~$\lambda_2$ as power series in~$\tau$. These power series converge
  on all of~$\ZZ_3$, so their coefficients tend to zero $3$-adically.
  By Hensel's Lemma, a simple zero mod~$3$ of such a power series will lift
  to a unique zero in~$\ZZ_3$. Now we extract the mod~$3$ part of the
  two power series. This gives us two polynomials in~$\FF_3[\tau]$.
  We check that their greatest common divisor is~$\tau$, both for $j = 1$
  and for~$j = 8$. For example, with our choice of uniformizer at~$P_1$
  we obtain
  \[ \lambda_1(\tau) \equiv \tau^3 + \tau^2 + \tau \bmod 3 \qquad\text{and}\qquad
     \lambda_2(\tau) \equiv \tau^3 + 2 \tau^2 + \tau \bmod 3 \,.
  \]
  Since the $\tau$-value of a rational point must be a zero of both
  $\lambda_1$ and~$\lambda_2$, we see that $\tau \equiv 0 \bmod 3$ and that
  we have a simple root mod~$3$ of at least one (and in fact, both) of the
  functions. This then implies that there is at most one common root
  in~$\ZZ_3$, which is taken care of by the known point~$P$. So $P$ is the only
  rational point in its residue class. This completes the proof.
\end{proof}



\section{L-series computations} \label{Sec: BSD}

In this section, we prove Theorem~\ref{Thm: BSD bound for the rank}, which states
that, assuming standard conjectures on L-series, we have $L'''(\jpc{4,0},1) \neq 0$.
In the end, this comes down to a numerical computation that can be performed
for example with Dokchitser's implementation of his method~\cite{Dokchitser}
in~\textit{Magma}. This procedure needs as input the conductor, the sign of
the functional equation and sufficiently many coefficients of the L-series.
The package also contains a function that checks numerically if the L-series
satisfies the functional equation with the given conductor and sign.

The L-series has an Euler product expansion $L(\jpc{4,0}, s) = \prod_p L_p(p^{-s})^{-1}$
with polynomials $L_p \in \ZZ[X]$ of degree at most~$10$.
When $p$ is a prime of good reduction,
then the degree is exactly~$10$, and the coefficient of $X^k$ in~$L_p$ can
be determined by counting the $\FF_{p^e}$-points on~$\cpc{4,0}$ for
$1 \le e \le \min\{k, 5\}$. More precisely, for certain algebraic integers~$\alpha_i$
we have
\[ L_p(X) = \prod_{i=1}^{10} (1 - \alpha_i X) \]
and
\[ \#\cpc{4,0}(\FF_{p^e}) = p^e + 1 - \sum_{i=1}^{10} \alpha_i^e \,, \]
and $L_p$ satisfies the functional equation $L_p(X) = p^5 X^{10} L_p(1/pX)$.
The point counts give us the power sums of the $\alpha_i$, which in turn
determine the elementary symmetric polynomials. If we know the coefficients
up to~$X^5$, then the functional equation of~$L_p$ provides us with the
remaining ones. So, at least in principle, we can compute all the necessary
coefficients of the Euler factors for primes of good reduction.

It remains to deal with the primes of bad reduction: $p = 2, 23, 2551$.
We can find the Euler factor and obtain information on the conductor by
constructing a regular proper model of the curve over~$\ZZ_p$. For other
examples of such computations, see \cite{FLSSSW} or~\cite{PSS}.

We deal with the bad primes in increasing order of difficulty. We begin
with $p = 2551$. The curve $\cpc{4,0}/\FF_{2551}$ has only one singularity,
which is a simple node with both tangent directions defined over~$\FF_{2551}$.
This point is regular on $\cpc{4,0}/\ZZ_{2551}$, so we already have a regular
model. The graph associated to the special fiber has one loop; Frobenius
acts trivially on the first homology. We see that the reduction of~$\jpc{4,0}$
has a 4-dimensional abelian and a 1-dimensional toric part. The Euler factor is
\[ L_{2551}(X) = (1 - X) \tilde{L}(X) = 1 + 185 X + \cdots \,, \]
where $\tilde{L}$ is the Euler factor of the smooth projective model
of~$\cpc{4,0}/\FF_{2551}$. (We do not need further coefficients of~$L_{2551}$
for the computation.) The exponent of the conductor at $p = 2551$ is~$1$.

Now consider $p = 23$. Here the reduction of~$\cpc{4,0}$ has three distinct
singularities, which are simple nodes with tangent directions defined
over~$\FF_{23}$. All three points are regular on~$\cpc{4,0}/\ZZ_{23}$.
The graph associated to the special fiber has three independent loops,
with Frobenius again acting trivially on the homology. We deduce that
the reduction of~$\jpc{4,0}$ has a 2-dimensional abelian and a 3-dimensional
toric part.
Let $\tilde{L}(X)$ be the Euler factor of the smooth projective model
of $\cpc{4,0}/\FF_{23}$ (a curve of genus~$2$), then
\[ L_{23}(X) = (1 - X)^3 \tilde{L}(X)
             = 1 + 10 X + 50 X^2 + 79 X^3 -123 X^4 - 776 X^5 + 1288 X^6 - 529 X^7 \,.
\]
The exponent of the conductor at $p = 23$ is~$3$.

Finally take $p = 2$. Here we have a single singularity on the reduction,
which is, however, a more complicated singularity than a simple node, and
it is a non-regular point on~$\cpc{4,0}/\ZZ_2$. So we blow it up. The
special fiber of the resulting model consists of the strict transform of the
original special fiber together with four lines meeting it in one point,
which is again not regular. So we now blow up this point, which produces
two further lines, each of multiplicity~$3$, meeting the original component
in a (now regular) point and each meeting two of the lines that showed up
at the previous stage at distinct regular points. If we label the four lines
that were obtained after the first blow-up by $A$, $B$, $C$, $D$, the
original component by~$E$ (which is now a smooth curve of genus~$0$)
and the two new lines (of multiplicity~$3$) by $F$ and~$G$ (see the configuration
sketched in Figure~\ref{fig1} on the right),
then we obtain the matrix of intersection numbers on the left in Figure~\ref{fig1}.

\begin{figure}[ht]
  \begin{center}
  \begin{minipage}{0.45\textwidth}
    \[
    \begin{array}{r|ccccccc|}
        &  A &  B &  C &  D &  E &  F &  G \\\hline
      A & -3 &  0 &  0 &  0 &  0 &  0 &  1 \\
      B &  0 & -3 &  0 &  0 &  0 &  1 &  0 \\
      C &  0 &  0 & -3 &  0 &  0 &  1 &  0 \\
      D &  0 &  0 &  0 & -3 &  0 &  0 &  1 \\
      E &  0 &  0 &  0 &  0 & -6 &  1 &  1 \\
      F &  0 &  1 &  1 &  0 &  1 & -2 &  1 \\
      G &  1 &  0 &  0 &  1 &  1 &  1 & -2 \\\hline
    \end{array}
    \]
  \end{minipage}
  \hfill
  \begin{minipage}{0.45\textwidth}
  \Gr{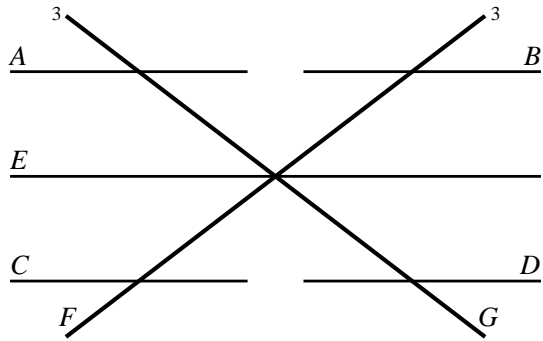}{\textwidth}  
  \end{minipage}
  \end{center}
  \caption{\label{fig1}
         Special fiber of the minimal regular model of $\cpc{4,0}$ over
         $\QQ_2$ (right) and the corresponding intersection matrix (left).}
\end{figure}

\medskip

From this, one may conclude that the group of connected components of the
special fiber of the N\'eron model of~$\jpc{4,0}$ over~$\ZZ_2$ is isomorphic
to~$\ZZ/3\ZZ \times \ZZ/21\ZZ$, but we will not need this information in what
follows. The configuration of components in the special fiber of the regular
model of~$\cpc{4,0}$ is tree-like and all components have genus~$0$, so
the reduction of~$\jpc{4,0}$ is totally unipotent, and the Euler factor
is $L_2(X) = 1$. We also see that the tame part of the conductor exponent
at~$p = 2$ is~$10$. We summarize the information we have obtained on the
conductor.

\begin{lem}
  The conductor of $\jpc{4,0}$ is $2^{10+w} \cdot 23^3 \cdot 2551$,
  where $w \ge 0$ is the wild part of the conductor at~$2$.
\end{lem}

Note that this is consistent with the fact that the elliptic curve~$E$ 
of conductor $92 = 2^2 \cdot 23$ occurs as a factor of $\jpc{4,0}$.
We can conclude that the simple 4-dimensional factor of~$\jpc{4,0}$
has conductor $2^{8+w} \cdot 23^2 \cdot 2551$.

In fact, we can say more.

\begin{prop} \label{Prop-cond}
  We have $w = 0$, so that the conductor of $\jpc{4,0}$ is
  $2^{10} \cdot 23^3 \cdot 2551$.
\end{prop}

\begin{proof}
  We show that $\jpc{4,0}$ acquires good reduction over $\QQ_2(\sqrt[21]{2})$,
  which is a tamely ramified extension of~$\QQ_2$. It follows that the
  $l$-division field $\QQ(\jpc{4,0}[l])$ of~$\jpc{4,0}$ is tamely ramified
  at~2, for any odd prime~$l$. This in turn implies that the wild part of
  the conductor vanishes, see for example~\cite{BrumerKramer}.
  
  We will establish the following assertions.
  
  \begin{lem} \label{Lemma-tame3}
    The special fiber of the minimal regular model of~$\cpc{4,0}$ over
    $\QQ_2(\sqrt[3]{2})$ contains two components of multiplicity~1 that
    are curves of genus~1.
  \end{lem}
  
  \begin{lem} \label{Lemma-tame7}
    The special fiber of the minimal regular model of~$\cpc{4,0}$ over
    $\QQ_2(\sqrt[7]{2})$ contains a component of multiplicity~1 that
    is a curve of genus~3.
  \end{lem}
  
  Since multiplicity-1 components of positive genus in the special fiber
  persist under field extensions, Lemma~\ref{Lemma-tame3} and
  Lemma~\ref{Lemma-tame7} together imply that the special fiber of the
  minimal regular model of~$\cpc{4,0}$ over~$\QQ_2(\sqrt[21]{2})$
  contains two curves of genus~1 and a curve of genus~3. Since the sum
  of their genera equals the genus of~$\cpc{4,0}$, we see that $\jpc{4,0}$
  has totally abelian, and therefore good, reduction over~$\QQ_2(\sqrt[21]{2})$.
\end{proof}

\begin{figure}[ht]
\begin{center}
  \Gr{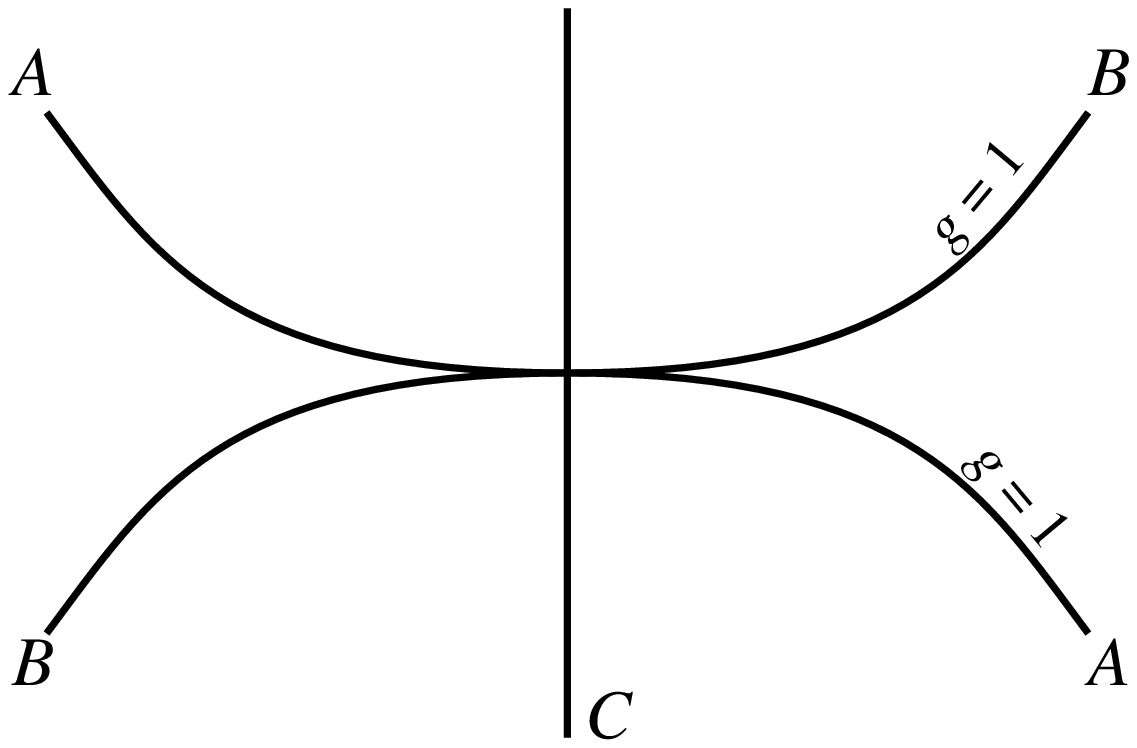}{0.45\textwidth} \hfill \Gr{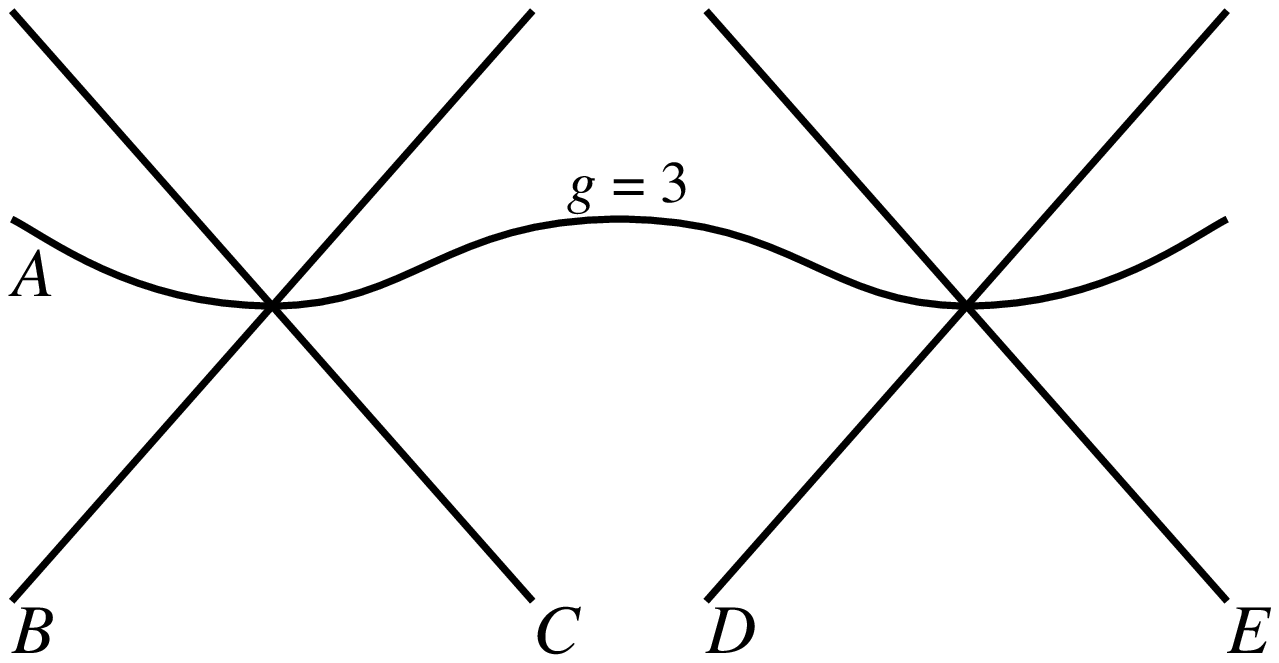}{0.45\textwidth}
\end{center}
\caption{\label{fig2}
         Special fibers of minimal regular models of $\cpc{4,0}$ over
         $\QQ_2(\sqrt[3]{2})$ (left) and $\QQ_2(\sqrt[7]{2})$ (right).}
\end{figure}

It remains to prove the two lemmas.

\begin{proof}[Proof of Lemma~\ref{Lemma-tame3}]
  In principle, this can be proved by computing the minimal regular model,
  which turns out to have a special fiber consisting of three components,
  all of multiplicity~1:
  two smooth curves of genus~1 meeting in one point with intersection
  multiplicity~3, and a smooth curve of genus~0 meeting them both
  transversally in their point of intersection, see Figure~\ref{fig2}, left.
  
  To simplify the verification of the claim, we give here another model
  of~$\cpc{4,0}$ over~$\QQ_2(\sqrt[3]{2})$ whose reduction contains the
  two genus-1 curves. (This is one of the charts one encounters when
  computing the successive blow-ups necessary to find a regular model.)
  We begin with the projective model
  \begin{align*}
    Z_1 Z_4 + Z_3^2 - Z_0^2 &= 0 \\
    Z_2 Z_4 + Z_3^2 - Z_1^2 &= 0 \\
    Z_3 Z_4 + Z_3^2 - Z_2^2 &= 0
  \end{align*}
  that was described in Section~\ref{Sec: Alternate}. 
  We dehomogenize by setting $Z_0 = 1$ and move
  the singularity to the origin by replacing $Z_1 \leftarrow z_1+1$,
  $Z_2 \leftarrow z_2+1$, $Z_3 \leftarrow z_3+1$ and $Z_4 \leftarrow z_4$.
  After subtracting
  the first equation from the other two, and including a redundant
  equation, we arrive at the following affine model.
  \begin{align*}
    z_1 z_4 + z_3^2 + 2 z_3 + z_4 &= 0 \\
    z_2 z_4 - z_1 z_4 - z_1^2 - 2 z_1 &= 0 \\
    z_3 z_4 - z_1 z_4 - z_2^2 - 2 z_2 &= 0 \\
    z_5 - \sqrt[3]{2} &= 0
  \end{align*}
  Note that the last equation implies $z_5^3 - 2 = 0$.
  We add $z_3$, $-z_1$, $-z_2$ times $z_5^3 - 2$, respectively,
  to the first three equations; then we substitute as follows:
  \begin{align*}
    z_1 &\leftarrow (x_1 x_3 + x_2)^2 x_2^4 x_3^3 x_4^3 \\
    z_2 &\leftarrow (x_1 x_3 + x_2)^2 x_2^3 x_3^3 x_4^2 \\
    z_3 &\leftarrow (x_1 x_3 + x_2) x_2^2 x_3^2 x_4^2 \\
    z_4 &\leftarrow (x_1 x_3 + x_2)^2 x_2^4 x_3^4 x_4^3 x_5 \\
    z_5 &\leftarrow (x_1 x_3 + x_2) x_2 x_3 x_4.
  \end{align*}
  We can then divide the four equations by
  \[ (x_1 x_3 + x_2)^2 x_2^4 x_3^4 x_4^3, \quad
     (x_1 x_3 + x_2)^4 x_2^7 x_3^6 x_4^5, \quad
     (x_1 x_3 + x_2)^3 x_2^6 x_3^6 x_4^4 \quad\text{and}\quad
     1,
  \]
  respectively. Let $e_1, e_2, e_3, e_4$ be the resulting polynomials,
  and let $e'_4 = (e_4 + \sqrt[3]{2})^3 - 2$ (this corresponds to $z_5^3 - 2$). 
  
  Now observe that
  \[ e'_2 = \frac{e_2 - x_3 e_1 - x_2 x_4 e'_4}{x_3 x_4} \]
  is again a polynomial; $e_1, e'_2, e_3, e_4$ define the desired affine model.
  Its reduction mod~$\sqrt[3]{2}$ contains the curves of genus~1 given by
  \[ x_2^2 x_5 + x_5^2 + x_2 = 0, \quad x_1 = 1 + x_2 x_5, \quad
     x_3 = 0, \quad x_4 = x_5
  \]
  and
  \[ x_3^2 x_5 + x_5^2 + x_3 = 0, \quad x_1 = 1, \quad x_2 = 0, \quad x_4 = x_5, \]
  respectively.
  
  Finally, we note that the above operations do indeed determine a birational map on generic fibers. For  $x_1 x_3 + x_2$, $x_2$, $x_3$ and $x_4$ must all be nonzero,
since otherwise $z_5$ would vanish, but it is equal to $\sqrt[3]{2}$. So to go back,
we can safely multiply the equations by powers of these expressions, and then
set
\begin{align*}
 x_1 &= z_1 (z_5^3 - z_1)/(z_2 z_3^2) \\
 x_2 &= z_1/(z_3 z_5) \\
 x_3 &= z_2 z_3/(z_1 z_5) \\
 x_4 &= z_3 z_5/z_2 \\
 x_5 &= z_4 z_5/(z_2 z_3),
 \end{align*}
and simplify.
\end{proof}

\begin{proof}[Proof of Lemma~\ref{Lemma-tame7}]
  This is analogous to the proof of Lemma~\ref{Lemma-tame3}. Here the
  special fiber of the minimal regular model over~$\QQ_2(\sqrt[7]{2})$
  has five components, all of multiplicity~1. One of them is a smooth
  genus-3 curve~$C$; the other four are smooth genus-0 curves, the first
  two of which meet $C$ pairwise transversally in one point, and the same
  is true of the remaining two, see Figure~\ref{fig2} on the right.
  
  As in the proof of Lemma~\ref{Lemma-tame3}, we make the genus-3 component
  explicit. This time, we dehomogenize by setting $Z_3 = 1$, then
  we eliminate $Z_4 = Z_2^2 - 1$. We obtain an affine model
  \[ Z_1 (Z_2^2 - 1) + 1 - Z_0^2 = Z_2 (Z_2^2 - 1) + 1 - Z_1^2 = 0 \,. \]
  We substitute $Z_0 \leftarrow z_0 + z_1 + z_2 + 1$,
  $Z_1 \leftarrow z_1 + z_2 + 1$, $Z_2 \leftarrow z_2 + 1$, and replace
  the first equation by its difference with the second. This gives the
  new model
  \[ -z_0^2 - 2 z_0 z_1 - 2 z_0 z_2 - 2 z_0 + z_1 z_2^2 + 2 z_1 z_2
      = -z_1^2 - 2 z_1 z_2 - 2 z_1 + z_2^3 + 2 z_2^2 = 0 \,.
  \]
  Now, writing $\pi = \sqrt[7]{2}$, we set $z_0 = \pi^7 x_0$, $z_1 = \pi^6 x_1$
  and $z_2 = \pi^4 x_2$. We can then divide the first equation by~$\pi^{14}$
  and the second by~$\pi^{12}$. The reductions mod~$\pi$ of the two equations
  are then
  \[  x_0 +  x_0^2 + x_1 x_2^2 = 0 \quad\text{and}\quad  x_1^2 + x_2^3 = 0 \,; \]
  they define a curve of (geometric) genus~3 (with two singular points at
  $x_1 = x_2 = 0$).
\end{proof}

\begin{remark}
  A priori, we know that $w$ is the wild part of the conductor of the
  simple 4-dimensional factor of~$\jpc{4,0}$. According to~\cite{BrumerKramer},
  the best general bound is then $w \le 40$. Using the result of
  Lemma~\ref{Lemma-tame3} and \cite[Prop.~6.11]{BrumerKramer}, we can
  reduce the bound to $w \le 22$. Using the result of Lemma~\ref{Lemma-tame7}
  only and the fact that the elliptic curve factor~$E$ still has additive
  reduction over~$\QQ_2(\sqrt[7]{2})$ (of type~IV), the bound drops
  to $w \le 6$. Using any of these bounds, one could prove that
  $w$ must vanish (under the assumption that the L-series has an analytic
  continuation and satisfies the correct functional equation) by checking
  numerically that the functional equation fails for any other possible
  value of~$w$. However, to do this kind of computation with reasonable
  precision, one needs an increasing number of L-series coefficients,
  which may become prohibitive for large~$w$.
\end{remark}

Using a sufficient number of L-series coefficients, we check that the
numerical criterion for the global functional equation is satisfied with
the conductor as established in Proposition~\ref{Prop-cond}
and sign~$-1$ as suggested by the parity of the
rank of the known subgroup of~$\jpc{4,0}(\QQ)$. As an additional check,
we performed the same computation with sign~$+1$, which clearly shows
that the functional equation does not hold with sign~$+1$.
If we assume that the L-series of~$\jpc{4,0}$ satisfies a
functional equation of the expected kind, then we must have sign~$-1$.

Having determined the conductor and the sign, we can compute the values
of the L-series and its higher derivatives at~$s = 1$. We obtain
\[ L(\jpc{4,0}, 1) \approx 0, \quad
   L'(\jpc{4,0}, 1) \approx 0, \quad
   L''(\jpc{4,0}, 1) \approx 0
\]
and
\[ L'''(\jpc{4,0}, 1) \approx 5.0846055622 \,. \]
This verifies that the third derivative does not vanish. 
The fact that the first derivative vanishes (within the precision of the
computation) as predicted by the conjecture of Birch and Swinnerton-Dyer
provides further corroboration that the conjctures we are assuming hold
in our case.
(The vanishing of the even derivatives is implied by the functional equation,
of course).


\medskip

The conjecture of Birch and Swinnerton-Dyer claims that the Mordell-Weil
rank of~$\jpc{4,0}$ is equal to the order of vanishing of~$L(\jpc{4,0},s)$
at~$s = 1$. Our computation shows that this order of vanishing is at most~$3$.
This gives an upper bound of~$3$ for the Mordell-Weil rank. On the other hand,
we already know that the rank is at least~$3$; see Theorem~\ref{thm_jac_properties}.
So the rank must be exactly~$3$, which implies by Theorem~\ref{Thm: Bound for X(4,0)}
 that $\kappa(0) = 6$.


\section*{Acknowledgments} We would like to thank the American Institute of Mathematics for hosting the workshop on ``The Uniform Boundedness Conjecture in Arithmetic Dynamics'' in January 2008. This project was conceived there.  We thank Michelle Manes for some of the early computations, 
Rob Benedetto for comments on an earlier draft of this paper, and the anonymous referee for several insightful comments. The first author was supported by a National Science Foundation Postdoctoral Research Fellowship during part of this work.

\bibliographystyle{plain}	
\bibliography{pre-image_zero}

\providecommand\biburl[1]{\texttt{#1}}
\begin{thebibliography}{10}

\bibitem{magma}
Wieb Bosma, John Cannon, and Catherine Playoust.
\newblock The {M}agma algebra system. {I}. {T}he user language.
\newblock {\em Journal of Symbolic Computation}, 24(3-4):235--265, 1997.

\bibitem{BruinStoll}
Nils Bruin and Michael Stoll.
\newblock The {M}ordell-{W}eil sieve: proving non-existence of rational points
  on curves.
\newblock {\em LMS Journal of Computation and Mathematics}, 13(-1):272--306,
  2010.

\bibitem{BrumerKramer}
Armand Brumer and Kenneth Kramer.
\newblock The conductor of an abelian variety.
\newblock {\em Compositio Math.}, 92(2):227--248, 1994.

\bibitem{Cremona_Data}
J.~E. Cremona.
\newblock {\em Elliptic {C}urve {D}ata}, 2008.
\newblock \verb+http://www.warwick.ac.uk/~masgaj/ftp/data/INDEX.html+.

\bibitem{CrePriSik}
J.~E. Cremona, M.~Prickett, and Samir Siksek.
\newblock Height difference bounds for elliptic curves over number fields.
\newblock {\em J. Number Theory}, 116(1):42--68, 2006.

\bibitem{Dokchitser}
Tim Dokchitser.
\newblock Computing special values of motivic {$L$}-functions.
\newblock {\em Experiment. Math.}, 13(2):137--149, 2004.

\bibitem{FHIJMTZ}
Xander Faber, Benjamin Hutz, Patrick Ingram, Rafe Jones, Michelle Manes,
  Thomas~J. Tucker, and Michael~E. Zieve.
\newblock Uniform bounds on pre-images under quadratic dynamical systems.
\newblock {\em Math. Res. Lett.}, 16(1):87--101, 2009.

\bibitem{FPS}
E.~V. Flynn, Bjorn Poonen, and Edward~F. Schaefer.
\newblock Cycles of quadratic polynomials and rational points on a genus-{$2$}
  curve.
\newblock {\em Duke Math. J.}, 90(3):435--463, 1997.

\bibitem{FLSSSW}
E.~Victor Flynn, Franck Lepr{\'e}vost, Edward~F. Schaefer, William~A. Stein,
  Michael Stoll, and Joseph~L. Wetherell.
\newblock Empirical evidence for the {B}irch and {S}winnerton-{D}yer
  conjectures for modular {J}acobians of genus 2 curves.
\newblock {\em Math. Comp.}, 70(236):1675--1697 (electronic), 2001.

\bibitem{Hindry_Silverman_book_2000}
Marc Hindry and Joseph~H. Silverman.
\newblock {\em Diophantine geometry. An Introduction}, volume 201 of {\em
  Graduate Texts in Mathematics}.
\newblock Springer-Verlag, New York, 2000.

\bibitem{Hutz_et_al}
Benjamin Hutz, Trevor Hyde, and Benjamin Krause.
\newblock Pre-images of quadratic dynamical systems.
\newblock \verb+arXiv:1007.0744 [math.NT]+, preprint, 2010.

\bibitem{ManinTorsion}
Ju.~I. Manin.
\newblock The {$p$}-torsion of elliptic curves is uniformly bounded.
\newblock {\em Izv. Akad. Nauk SSSR Ser. Mat.}, 33:459--465, 1969.

\bibitem{Mazur_Torsion_1978}
B.~Mazur.
\newblock Modular curves and the {E}isenstein ideal.
\newblock {\em Inst. Hautes \'Etudes Sci. Publ. Math.}, (47):33--186 (1978),
  1977.

\bibitem{Poonen}
Bjorn Poonen.
\newblock The classification of rational preperiodic points of quadratic
  polynomials over {${\bf Q}$}: a refined conjecture.
\newblock {\em Math. Z.}, 228(1):11--29, 1998.

\bibitem{PSS}
Bjorn Poonen, Edward~F. Schaefer, and Michael Stoll.
\newblock Twists of {$X(7)$} and primitive solutions to {$x^2+y^3=z^7$}.
\newblock {\em Duke Math. J.}, 137(1):103--158, 2007.

\bibitem{StollHeights1}
Michael Stoll.
\newblock On the height constant for curves of genus two.
\newblock {\em Acta Arith.}, 90(2):183--201, 1999.

\bibitem{Stoll2Descent}
Michael Stoll.
\newblock Implementing 2-descent for {J}acobians of hyperelliptic curves.
\newblock {\em Acta Arith.}, 98(3):245--277, 2001.

\bibitem{StollHeights2}
Michael Stoll.
\newblock On the height constant for curves of genus two. {II}.
\newblock {\em Acta Arith.}, 104(2):165--182, 2002.

\bibitem{Stoll_Independence_2006}
Michael Stoll.
\newblock Independence of rational points on twists of a given curve.
\newblock {\em Compos. Math.}, 142(5):1201--1214, 2006.

\bibitem{Stoll3}
Michael Stoll.
\newblock Rational 6-cycles under iteration of quadratic polynomials.
\newblock {\em LMS J. Comput. Math.}, 11:367--380, 2008.

\bibitem{vanluijk}
Ronald van Luijk.
\newblock K3 surfaces with {P}icard number one and infinitely many rational
  points.
\newblock {\em Algebra Number Theory}, 1(1):1--15, 2007.

\end{thebibliography}

\end{document}